\documentclass{amsart}[12pt]
\usepackage{amssymb,amsmath}
\usepackage{color}

\newtheorem{theorem}{Theorem}[section]
\newtheorem{proposition}{Proposition}[section]
\newtheorem{corollary}{Corollary}[section]
\newtheorem{lemma}{Lemma}[section]

\theoremstyle{definition}
\newtheorem{remark}{Remark}[section]
\newtheorem{definition}{Definition}[section]
\newtheorem{example}{Example}[section]

\title[Noncommutative Subsequential Weighted Ergodic Theorems]{Some noncommutative subsequential weighted individual ergodic theorems}
\keywords{Semifinite von Neumann algebra, noncommutative ergodic theorem, positive density sequence, block sequence, uniform sequence, bounded Besicovitch sequence}
\subjclass[2020]{47A35, 46L52}

\author{Morgan O'Brien}
\address{North Dakota State University\\ Department of Mathematics\\ 1210 Albrecht Boulevard, Minard Hall \\ Fargo, ND 58102, USA}
\email{morgan.obrien@ndsu.edu, obrienmorganc@gmail.com}

\begin{document}
\begin{abstract}
This article is devoted to studying individual ergodic theorems for subsequential weighted ergodic averages on the noncommutative $L_p$-spaces associated to a semifinite von Neumann algebra $\mathcal{M}$.
In particular, we establish the convergence of these averages along sequences with density one and certain types of block sequences with positive lower density, and we extend known results along uniform sequences in the sense of Brunel and Keane.
\end{abstract}
\date{March 30, 2021}

\maketitle
\section{Introduction}

Beginning with the subsequential ergodic theorems of Blum-Hanson (for norm convergence \cite{bh}) and Brunel-Keane (for a.e. convergence \cite{bk}), proving the norm or a.e. convergence of subsequential or weighted averages has been an active focus of research in ergodic theory.  Some of these results have been extended to the von Neumann algebra setting.

Let $\mathcal{M}$ be a semifinite von Neumann algebra acting on a Hilbert space $\mathcal{H}$ (containing the identity operator, denoted by $1$) with a  normal semifinite faithful (n.s.f.) trace $\tau$, and let $L_1(\mathcal{M},\tau)$ be the space of $\tau$-measurable operators affiliated to $\mathcal{M}$ such that $\|x\|_1:=\tau(|x|)<\infty$.

Let $T:L_1(\mathcal{M},\tau)+\mathcal{M}\to L_1(\mathcal{M},\tau)+\mathcal{M}$ be a positive linear map which contracts $L_1(\mathcal{M},\tau)$ and $\mathcal{M}$. If $L_p(\mathcal{M},\tau)$ is the associated noncommutative $L_p$-space to $\mathcal{M}$, then $T$ extends to a positive linear contraction on $L_p(\mathcal{M},\tau)$ for every $1\leq p\leq\infty$.

Given $1\leq p\leq\infty$ and $x\in L_p(\mathcal{M},\tau)$, let $$M_n(T)(x)=\frac{1}{n}\sum_{j=0}^{n-1}T^j(x)$$ be the $n$-th \textit{ergodic average} of $x$ associated to $T$. If $\textbf{k}=\{k_j\}_{j=0}^{\infty}$ is a strictly increasing sequence of nonnegative integers, let $$M_n^{\textbf{k}}(T)(x)=\frac{1}{n}\sum_{j=0}^{n-1}T^{k_j}(x)$$denote the \textit{$n$-th subsequential ergodic average} of $x$ associated to $T$ and $\textbf{k}$, or just the $n$-th subsequential average for short. 

In \cite{la} it was proven that the averages $M_n(T)(x)$ converge almost uniformly (a.u.) when $T$ is an automorphism on $\mathcal{M}$, $\tau$ is a faithful $T$-invariant state, and $x\in\mathcal{M}$. It was shown in \cite{ye} that $\tau$ can be replaced by a n.s.f. trace (i.e. not necessarily a state), $x\in\mathcal{M}$ by $x\in L_1(\mathcal{M},\tau)$, and $T$ by a positive Dunford Schwartz operator, at the price of the convergence holding bilaterally almost uniformly (b.a.u.). This result was then extended in \cite{jx} to hold on $L_p(\mathcal{M},\tau)$ for each $1\leq p<\infty$, and to hold a.u. whenever $2\leq p<\infty$. Alternative proofs of both of these results are given in \cite{li} using the notions of uniform and bilaterally uniform equicontinuity in measure at zero.
In \cite{lm} it was proven that the subsequential ergodic averages $\{M_n^{\textbf{k}}(T)_n(x)\}_{n=1}^{\infty}$ associated to a uniform sequence $\{k_j\}_{j=0}^{\infty}$ (see Section 4) converge b.a.u. for every $x\in L_1(\mathcal{M},\tau)$.

If $\beta=\{\beta_k\}_{k=0}^{\infty}$ is a sequence of complex numbers, then $$M_n^\beta(T)(x)=\frac{1}{n}\sum_{k=0}^{n-1}\beta_kT^k(x)$$ will be called the $n$-th \textit{weighted ergodic average} of $x$, and the $n$-th \textit{subsequential weighted ergodic average} of $x$ is defined as $$M_n^{\beta,\textbf{k}}(T)(x)=\frac{1}{n}\sum_{j=0}^{n-1}\beta_{k_j}T^{k_j}(x).$$

In this article, we will prove that the subsequential weighted averages are b.u.e.m. at zero on $(L_p(\mathcal{M},\tau),\|\cdot\|_p)$ for each $1\leq p<\infty$ and u.e.m. at zero for $2\leq p<\infty$ when $\{k_j\}_{j=0}^{\infty}$ is a sequence with positive lower density and $\{\beta_j\}_{j=0}^{\infty}$ is bounded (where we write u.e.m. or b.u.e.m. at zero for uniform or bilaterally uniform equicontinuity in measure at zero, which are given in detail in Definition 2.1). Utilizing this fact, we will prove that the subsequential weighted averages along any sequence of density one converge a.u. (b.a.u.) in $L_p$ for $2\leq p<\infty$ ($1\leq p<\infty$) for weights which work for the full sequence. After that, we extend the results of Litvinov and Mukhamedov by showing that the subsequential weighted averages along a uniform sequence converge a.u. (b.a.u.) in $L_p$ for $2\leq p<\infty$ ($1\leq p<\infty$) for certain weights. Finally, we prove that the subsequential weighted averages along certain types of block sequences with positive density converge a.u. (b.a.u.) in $L_p$ for $2\leq p<\infty$ ($1\leq p<\infty)$ when weighted by a bounded Besicovich sequence.

\section{Preliminaries}

We will begin by providing a brief review of the notation and properties of noncommutative integration theory that will be used in the sequel. For more in depth details and proofs of these results, see \cite{fk}, \cite{ne}, or \cite{se}.

Let $\mathcal{B}(\mathcal{H})$ be the space of bounded linear operators on $\mathcal{H}$, with the operator norm $\|\cdot\|_\infty$. Let $x:\mathcal{D}_x\to \mathcal{H}$ be a closed densely defined linear operator on $\mathcal{H}$, where $\mathcal{D}_x\subseteq \mathcal{H}$ is the domain of $x$. Then $x$ is said to be \textit{affiliated} with $\mathcal{M}$ if $yx\subseteq xy$ for every $y\in \mathcal{M}'$, where $\mathcal{M}'$ is the commutant $\{y\in \mathcal{B}(\mathcal{H}):zy=yz\text{ for every }z\in\mathcal{M}\}$ of $\mathcal{M}$. Let $\mathcal{P}(\mathcal{M})=\{e\in \mathcal{M}:e=e^*=e^2\}$ denote the set of projections in $\mathcal{M}$, and if $e\in \mathcal{P}(\mathcal{M})$, we will write $e^\perp=1-e$, where $1$ is the identity of $\mathcal M$.

A closed densely defined operator $x$ affiliated with $\mathcal{M}$ is called $\tau$\textit{-measurable} if for every $\epsilon>0$ there exists $e\in\mathcal{P}(\mathcal{M})$ such that $\tau(e^\perp)\leq\epsilon$ and $e(\mathcal{H})\subseteq\mathcal{D}_x$. Let $L_0(\mathcal{M},\tau)$ denote the space of $\tau$-measurable operators affiliated with $\mathcal{M}$.

For every $\epsilon,\delta>0$, let
\[
V(\epsilon,\delta)=\{x\in L_0(\mathcal{M},\tau):\|xe\|_\infty\leq\epsilon\text{ for some }e\in \mathcal{P}(\mathcal{M})\text{ with }\tau(e^\perp)\leq\delta\}.
\] 
The \textit{measure topology} on $L_0(\mathcal{M},\tau)$ is the topology generated by the system of neighborhoods of zero $\{V(\epsilon,\delta):\epsilon,\delta>0\}$. With this, $L_0(\mathcal{M},\tau)$ is a complete metrizable topological $*$-algebra containing $\mathcal{M}$ as a dense subspace \cite{ne}.

If $\mathcal{S}\subseteq L_0(\mathcal{M},\tau)$, then we will let $\mathcal{S}^+=\{x\in\mathcal{S}:x\geq0\}$.
If $x$ is a positive, self-adjoint operator affiliated with $\mathcal{M}$ and $x=\int_{0}^{\infty}\lambda de_{\lambda}$ is its spectral decomposition, then
$$
\tau(x)=\sup_{n\geq1}\tau\Big(\int_{0}^{n}\lambda de_{\lambda}\Big).
$$
Given $1\leq p<\infty$, we will let $$L_p(\mathcal{M},\tau)=\{x\in L_0(\mathcal{M},\tau):\tau(|x|^p)<\infty\},$$ where $|x|=(x^*x)^{1/2}$ is the \textit{absolute value} of $x$. Define $\|x\|_p=\tau(|x|^p)^{1/p}$ whenever $x\in L_p(\mathcal{M},\tau)$ and $1\leq p<\infty$. We will let $L_\infty(\mathcal{M},\tau)=\mathcal{M}$, equipped with the operator norm $\|\cdot\|_\infty$. Then $\|\cdot\|_p$ is a norm on $L_p(\mathcal{M},\tau)$ which turns it into a Banach space whenever $1\leq p\leq\infty$. Note that $L_1\cap\mathcal{M}$ and $L_p\cap\mathcal{M}$ are both dense in $L_p(\mathcal{M},\tau)$ for $1\leq p<\infty$ with respect to the norm $\|\cdot\|_p$.

The semifinite von Neumann algebra ``of interest'' throughout this paper will always be denoted by $\mathcal{M}$ and its n.s.f. trace by $\tau$. As such, we will often write $L_p=L_p(\mathcal{M},\tau)$ to ease notation whenever $1\leq p<\infty$ or $p=0$ (and when confusion will not occur). 

We say that $\{x_n\}_{n=1}^{\infty}\subset L_0$ converges to $x\in L_0$ \textit{almost uniformly} (\textit{a.u.}) (respectively, \textit{bilaterally almost uniformly} (\textit{b.a.u.})) if for every $\epsilon>0$ there exists $e\in \mathcal{P}(\mathcal{M})$ such that 
\[
\tau(e^\perp)\leq\epsilon\text{ \ and \ } \|(x_n-x)e\|_\infty\to0\ (\text{resp., }\|e(x_n-x)e\|_\infty\to0)\text{ as }n\to\infty.
\]
Clearly a.u. convergence implies b.a.u. convergence; however, in the commutative case a.u. and b.a.u. convergences are equivalent. As shown in Example 3.1 of \cite{cl2}, this is not the case in general.

\begin{definition}\label{d21} Let $(X,\|\cdot\|)$ be a normed space. A sequence of maps $A_n: X\to L_0$, $n=1,2,\dots$, is said to be \textit{bilaterally uniformly equicontinous in measure} (\textit{b.u.e.m.}) (\textit{uniformly equicontinuous in measure} (\textit{u.e.m.})) at zero if for every $\epsilon>0$ and $\delta>0$ there exists $\gamma>0$ such that, given $x\in X$ with $\|x\|<\gamma$, there exists $e\in\mathcal{P}(\mathcal{M})$ such that
\[
\tau(e^\perp)\leq\epsilon, \ \text{ and } \ \sup_{n}\|eA_n(x)e\|_\infty\leq\delta\ (\text{resp.,\ }\sup_{n}\|A_n(x)e\|_\infty\leq\delta).
\]
\end{definition}

A linear map $T:L_1+\mathcal{M}\to L_1+\mathcal{M}$ such that
\[
\|T(x)\|_1\leq\|x\|_1\ \ \forall \ x\in L_1\ \text{ and }\ \|T(x)\|_\infty\leq\|x\|_\infty\ \ \forall \ x\in\mathcal{M}
\]
is called a \textit{Dunford-Schwartz operator}. If $T(x)\geq0$ whenever $x\geq0$, then $T$ is called a \textit{positive Dunford-Schwartz operator}, and we write $T\in DS^+(\mathcal{M},\tau)$, or just $T\in DS^+$.

Note that, as was shown in Proposition 1.1 of \cite{cl1}, absolute contractions considered in \cite{ye}, and later in \cite{lm}, can be uniquely extended to positive Dunford-Schwartz operators. Also, it was shown in Lemma 1.1 of \cite{jx} that any $T\in DS^+$ can be extended uniquely to a positive linear contraction on $L_p$ for each $1<p<\infty$. Thus, denoting these extensions by $T$, we have $\|T(x)\|_p\leq\|x\|_p$ for all $x\in L_p$.

Given $T\in DS^+$ and $1\leq p<\infty$, the positive linear maps $M_n(T):L_p\to L_p$ are continuous, and are also continuous when viewed as maps from $L_p$ to $L_0$ when $L_p$ has its norm topology and $L_0$ has the topology of convergence in measure.

Results regarding the u.e.m. and b.u.e.m. at zero of $\{M_n(T)\}_{n=1}^{\infty}$ above were established in \cite{li} (as Propositions 4.1, 4.2, 4.3 and Theorem 2.1), and Proposition 2.2 below is Theorem 2.1 of \cite{li}. We state these results here for ease of reference. 

\begin{proposition}\label{p21} The sequence $\{M_n(T)\}_{n=1}^{\infty}$ is b.u.e.m. (u.e.m.) at zero on $(L_p,\|\cdot\|_p)$ if $1\leq p<\infty$ (respectively $2\leq p<\infty$). 
\end{proposition}

\begin{proposition}\label{p22} Let $(X,\|\cdot\|)$ be a Banach space, $A_n:X\to L_0$ be a sequence of linear maps that is b.u.e.m. (u.e.m.) at zero on $X$. Then the set
$$C=\{x\in X:\{A_n(x)\}\text{ converges b.a.u. (a.u.)}\}$$
is closed in $X$.
\end{proposition}

The notions of u.e.m. and b.u.e.m. at zero were used in \cite{li} to simplify the ending of the proof of b.a.u. convergence in the individual ergodic theorem on $L_1$ from Theorem 2 of \cite{ye} and to provide alternative proofs of the extension of that result to noncommutative $L_p$-spaces (first proved in \cite{jx}).

\smallskip

\begin{theorem}\label{t21}\cite{ye,jx}. Let $\mathcal{M}$ be a semifinite von Neumann algebra with a normal semifinite faithful trace $\tau$, and let $T\in DS^+$. If $1\leq p<\infty$, then the averages $M_n(T)(x)$ converge b.a.u. to some $\widehat{x}\in L_p(\mathcal{M},\tau)$ for every $x\in L_p(\mathcal{M},\tau)$. Furthermore, this convergence occurs a.u. when $2\leq p<\infty$.
\end{theorem}

\section{Convergence along Sequences of Density One}

For clarity, $\mathbb{N}$ will denote the set of natural numbers, $\mathbb{N}_0=\mathbb N\cup\{0\}$, and $\mathbb C$ will stand for the set of complex numbers. Let $\chi_E:X\to\{0,1\}$ be the characteristic function of a set $E\subseteq X$.

For convenience, we will let $\textbf{1}=\{1\}_{j=0}^{\infty}$; in what follows, $\beta=\textbf{1}$ will correspond to the `non-weighted' case of the weighted results. Also, we will also identify a sequence $\textbf{k}=\{k_j\}_{j=0}^{\infty}$ of nonnegative integers with its image $\{k_j:j\in\mathbb{N}_0\}$ in $\mathbb{N}_0$.

\begin{definition}\label{d31} Let $T\in DS^+(\mathcal{M},\tau)$ and $1\leq p<\infty$. A sequence $\beta\subset\mathbb{C}$ will be called \textit{a.u.- (b.a.u.-) good for $T$ on $L_p(\mathcal{M},\tau)$} if the weighted averages $\{M_n^\beta(T)(x)\}_{n=1}^{\infty}$ converge a.u. (respectively, b.a.u.) for every $x\in L_p(\mathcal{M},\tau)$. Similarly, we will say that a strictly increasing sequence $\textbf{k}\subset\mathbb{N}_0$ is \textit{a.u.- (b.a.u.-) good for $T$ and $\beta$ on $L_p(\mathcal{M},\tau)$} if the subsequential weighted averages $\{M_n^{\beta,\textbf{k}}(T)(x)\}_{n=1}^{\infty}$ converge a.u. (respectively b.a.u.) for every $x\in L_p(\mathcal{M},\tau)$.

A sequence $\beta$ will be called \textit{a.u.- (b.a.u.-) good in $L_p(\mathcal{M},\tau)$} if $\beta$ is a.u.- (b.a.u.-) good for $T$ in $L_p(\mathcal{M},\tau)$ for every $T\in DS^+(\mathcal{M},\tau)$. Similarly, define $\textbf{k}$ as \textit{a.u.- (b.a.u.-) good for $\beta$ in $L_p(\mathcal{M},\tau)$} if it is a.u.- (b.a.u.-) good for $T$ and $\beta$ in $L_p(\mathcal{M},\tau)$ for every $T\in DS^+(\mathcal{M},\tau)$.
\end{definition}

\begin{example}\label{e32}
As stated in Theorem 2.1, for each von Neumann algebra $\mathcal{M}$ with n.s.f. trace $\tau$, the sequence $\beta=\textbf{1}$ is a.u.- (b.a.u.-) good in $L_p(\mathcal{M},\tau)$ for every $2\leq p<\infty$ (respectively, $1\leq p<\infty$).
\end{example}

\begin{example}\label{e31}
Let $\mathbb{T}=\{z\in\mathbb{C}:|z|=1\}$ denote the unit circle in $\mathbb{C}$. A function $P:\mathbb{Z}\to\mathbb{C}$ is called a \textit{trigonometric polynomial} if there exists $\{r_j\}_{j=1}^{s}\subset\mathbb{C}$ and $\{\lambda_j\}_{j=1}^{s}\subset\mathbb{T}$ such that $P(k)=\sum_{j=1}^{s}r_j\lambda_j^k$ for every $k\in\mathbb{Z}$.

\noindent A sequence $\{\beta_k\}_{k=0}^{\infty}\subset\mathbb{C}$ is called a \textit{bounded Besicovich sequence} if:

(i) $\sup_{k}|\beta_k|<\infty$

(ii) For any $\epsilon>0$, there exists a trigonometric polynomial $P$ such that
$$\limsup_{n\to\infty}\frac{1}{n}\sum_{k=0}^{n-1}|\beta_k-P(k)|<\epsilon.$$

Assume that $\mathcal{M}$ has a separable predual. It was shown in Theorem 4.6 of \cite{cls} that a bounded Besicovich sequence $\beta$ is b.a.u.-good in $L_1(\mathcal{M},\tau)$. Furthermore, it was shown in Theorem 3.1 of \cite{cl1} that $\beta$ is b.a.u.-good in $L_p(\mathcal{M},\tau)$  if $1\leq p<\infty$ and a.u.-good in $L_p(\mathcal{M},\tau)$ whenever $2\leq p<\infty$.
\end{example}

\begin{example}\label{e33}
Let $\textbf{k}$ be a uniform sequence in the sense of Brunel and Keane (details given in Section 4). If $\mathcal{M}$ has a separable predual, then it was shown in Theorem 4 of \cite{lm} that $\textbf{k}$ is b.a.u.-good for $\beta=\textbf{1}$ in $L_1(\mathcal{M},\tau)$. This can be rephrased as saying that $c=\{\chi_{\textbf{k}}(j)\}_{j=0}^{\infty}$ is b.a.u.-good in $L_1(\mathcal{M},\tau)$; see  Proposition \ref{p32} below.
\end{example}

\begin{example}\label{e34}
A sequence $\alpha=\{\alpha_k\}_{k=0}^{\infty}\subseteq\mathbb{C}$ will be said to have correlation if
$$\gamma_{\alpha}(m)=\lim_{n\to\infty}\frac{1}{n+1}\sum_{k=0}^{n}\overline{\alpha_k}\alpha_{k+m}$$ exists for every $m\in\mathbb{N}_0$. Also, set $\gamma_\alpha(-m)=\overline{\gamma(m)}$ for $m\in\mathbb{N}$. Suppose that $\alpha$ has correlation, is bounded, and satisfies the following conditions:

(i) the measure $\sigma_\alpha$ on $\mathbb{T}$ associated to $\gamma_\alpha$ given by the Bochner-Herglotz Spectral Theorem is discrete, where
$$\gamma_\alpha(m)=\int_{\mathbb{T}}\lambda^m\, d\sigma_\alpha(\lambda),\, \, m\in\mathbb{Z};$$

(ii) the limit $\lim_{n\to\infty}\frac{1}{n+1}\sum_{k=0}^{n}\alpha_k\lambda^k$ exists for all $\lambda\in\mathbb{T}$.

Suppose that $\tau(1)<\infty$ and that $T$ is a normal $\tau$-preserving $*$-homomorphism on $\mathcal{M}$. Then, by Theorem 1.3 of \cite{hs}, $\alpha$ is a.u.- (b.a.u.-) good for $T$ on $L_p(\mathcal{M},\tau)$ for $2\leq p<\infty$ (respectively, for $1\leq p<\infty$).

In fact, \cite{hs} proves a stronger result. Namely, for every $\epsilon>0$ there exists $e\in\mathcal{P}(\mathcal{M})$ such that $\tau(e^\perp)\leq\epsilon$ and the sequence $\{M_n^\alpha(T)(x)e\}_{n=1}^{\infty}$ ( $\{eM_n^\alpha(T)(x)e\}_{n=1}^{\infty}$) converges in $\mathcal{M}$ for \textit{every} sequence $\alpha$ as above when $x\in L_2$ (respectively, $x\in L_1$). Furthermore, it was shown that this type of sequence strictly contains all bounded Besicovich sequences.
\end{example}

Recall that $E\subset\mathbb{N}_0$ is said to have \textit{density (lower density) $d$} if
$$\lim_{n\to\infty}\frac{\operatorname{card}(\{0,...,n\}\cap E)}{n+1}=d \text{ \ \ (respectively, } \liminf_{n\to\infty}\frac{\operatorname{card}(\{0,...,n\}\cap E)}{n+1}=d).$$
Accordingly, a sequence $\textbf{k}=\{k_j\}_{j=0}^{\infty}$ of nonnegative integers has \textit{density (lower density)} $d$ if the set $\textbf{k}=\{k_j:j\in\mathbb{N}_0\}$ has density (respectively, lower density) $d$.

\begin{remark}\label{r31}
Note that, if $\textbf{k}$ has density $d>0$, then $\lim_{n\to\infty}\frac{k_n}{n}=\frac{1}{d}$. More generally, by Lemma 40 of \cite{ro}, $\textbf{k}$ having lower density $d>0$ is equivalent to the condition $\sup_{n\geq1}\frac{k_n}{n}<\infty$.
\end{remark}

In what follows, $T\in DS^+$ and $\textbf{k}=\{k_j\}_{j=0}^{\infty}\subset\mathbb{N}_0$ is a strictly increasing sequence.

\begin{proposition}\label{p31}
Let $\beta=\{\beta_j\}_{j=0}^{\infty}\in\ell_\infty(\mathbb{C})$. If $\textbf{k}=\{k_j\}_{j=0}^{\infty}$ has lower density $d>0$, then the sequences $\{M_n^\beta(T)\}_{n=1}^{\infty}$ and $\{M_n^{\beta,\textbf{k}}(T)\}_{n=1}^{\infty}$ are b.u.e.m. at zero on $(L_p,\|\cdot\|_p)$ for each $1\leq p<\infty$ and u.e.m. at zero on $(L_p,\|\cdot\|_p)$ for each $2\leq p<\infty$.
\end{proposition}
\begin{proof}
We will prove that the sequences are b.u.e.m. at zero for $1\leq p<\infty$. The u.e.m. at zero claim follows similarly by using Theorem 2.3 of \cite{cl1} below instead of Theorem 2.1 of \cite{cl1}.

We will first prove the assertion for $\{M_n^\beta(T)\}_{n=1}^{\infty}$. To that end, fix $\epsilon,\delta>0$. Since 
$\beta\in\ell_\infty(\mathbb{C})$, there exists $C>0$ such that $|\beta_j|\leq C$ for each $j\geq0$.

Fix $1\leq p<\infty$, and assume $x\in L_p$. By the maximal ergodic theorem given in Theorem 2.1 of \cite{cl1}, there exists $e\in \mathcal{P}(\mathcal{M})$ with
$$\tau(e^{\perp})\leq4\left(\frac{48C\|x\|_p}{\delta}\right)^p\ \text{ and }\ \sup_{n\geq1}\|eM_n^\beta(x)e\|_\infty\leq\delta.$$
Let $\gamma=\dfrac{\epsilon^{1/p}\delta}{4^{1/p}\cdot 48C}$ and assume further that $\|x\|_p<\gamma$. Then we still have $\sup_{n\geq1}\|eM_n^\beta(x)e\|_\infty\leq\delta$, while
\[
\tau(e^{\perp})\leq\left(\frac{4^{1/p}\cdot48C}{\delta}\|x\|_p\right)^p\leq\epsilon,
\]
implying that $\{M_n^\beta(T)\}_{n=1}^{\infty}$ is b.u.e.m. at zero on $(L_p,\|\cdot\|_p)$.

For each $j\geq0$, let $c_j=\chi_\textbf{k}(j)$. Then $|c_j|\leq1$ for every $j\geq0$, so $c\beta=\{c_j\beta_j\}_{j=0}^{\infty}$, is a bounded sequence. If $x\in L_p(\mathcal{M},\tau)$, we have
\[
\begin{split}
M_n^{\beta,\textbf{k}}(T)(x)&=\frac{1}{n}\sum_{j=0}^{n-1}\beta_{k_j}T^{k_j}(x)=\frac{k_{n-1}+1}{n}\frac{1}{k_{n-1}+1}\sum_{j=0}^{k_{n-1}}c_j\beta_jT^{j}(x)\\
&=\frac{k_{n-1}+1}{n}M_{k_{n-1}+1}^{c\beta}(T)(x).
\end{split}
\]
Replacing $\beta$ in the first part of the proof with $c\beta$, we see that $\{M_n^{c\beta}(T)\}_{n=1}^{\infty}$ is b.u.e.m. at zero on $(L_p,\|\cdot\|_p)$.

In view of Remark \ref{r31}, let $K=\sup_{n\geq1}\frac{k_n}{n}$. Assume $\epsilon,\delta>0$, and let $\gamma>0$ be the associated value for $\{M_n^{c\beta}(T)\}_{n=1}^{\infty}$ in the definition of b.u.e.m. at zero with respect to $\epsilon$ and $\frac{\delta}{K}$. Let $x\in L_p(\mathcal{M},\tau)$ be such that $\|x\|_p<\gamma$. Then there exists $e\in \mathcal{P}(\mathcal{M})$ such that $\tau(e^{\perp})\leq\epsilon$ and $\sup_{n\geq1}\|eM_n^{c\beta}(T)(x)e\|_\infty\leq\frac{\delta}{K}$. With this projection $e$,
we also find that
\begin{align*}
\sup_{n\geq1}\|eM_n^{\beta,\textbf{k}}(T)(x)e\|_\infty
&=\sup_{n\geq1}\left\{\frac{k_{n-1}+1}{n}\,\|eM_{k_{n-1}+1}^{c\beta}(T)(x)e\|_\infty\right\}\\
&\leq  K\sup_{n\geq1}\|eM_{n}^{c\beta}(T)(x)e\|_\infty \leq K\frac{\delta}{K}=\delta,
\end{align*}
hence $\{M_n^{\beta,\textbf{k}}(T)\}_{n=1}^{\infty}$ is b.u.e.m. at zero on $(L_p,\|\cdot\|_p)$.
\end{proof}

\medskip

The proof of the following lemma can be found in Lemma 3 of \cite{lm}, where, the result is proven for the b.a.u. convergence case, but a close examination of the proof shows that the same proof will also work for a.u. convergence with appropriate modifications.

\begin{lemma}\label{l31}
If a sequence $\{x_n\}_{n=1}^{\infty}$ in $\mathcal{M}$ is such that for every $\epsilon>0$ there are an a.u. (b.a.u.) convergent sequence $\{y_n\}_{n=1}^{\infty}\subset \mathcal{M}$ and a positive integer $N$ satisfying $\|x_n-y_n\|_\infty<\epsilon$ for all $n\geq N$, then $\{x_n\}_{n=1}^{\infty}$ converges a.u. (respectively, b.a.u.).
\end{lemma}

We will also need the following simple fact.

\begin{lemma}\label{l32}
If $\mathbb{C}\ni\lambda_n\to\lambda$ and $L_0\ni x_n\to\widehat{x}\in L_0$ a.u. (b.a.u.), then $\lambda_n x_n\to\lambda\widehat{x}$ a.u. (respectively, b.a.u.)
\end{lemma}
\begin{proof}
Assume $\epsilon>0$. Let $e\in\mathcal{P}(\mathcal{M})$ be such that $\tau(e^{\perp})\leq\epsilon$ and $\|(x_n-\widehat{x})e\|_\infty\to0$. It then follows that the sequence $\{\|x_ne\|_\infty\}_{n}$ is bounded, and we obtain
\[\begin{split}
\|(\lambda_nx_n-\lambda\widehat{x})e\|_\infty
&\leq\|(\lambda_n-\lambda)x_ne\|_\infty+\|\lambda(x_n-\widehat{x}e)\|_\infty \\
&=|\lambda_n-\lambda|\|x_ne\|_\infty+|\lambda|\|(x_n-\widehat{x})e\|_\infty\to0
\end{split}\]
as $n\to\infty$.
\end{proof}

The following is a noncommutative extension of Proposition 1.7 in \cite{bl}.

\begin{proposition}\label{p32} Assume $\{k_j\}_{j=0}^{\infty}$ has density $d>0$, and let $\beta=\{\beta_j\}_{j=0}^{\infty}\subset\mathbb{C}$. Define $\gamma_i=1$ if $i=k_j$ for some $j$ and $c_i=0$ otherwise, and denote $\gamma\beta=\{\gamma_j\beta_j\}_{j=0}^{\infty}$. Let $1\leq p<\infty$, $T\in DS^+$, and $x\in L_p$. Then the averages $M_n^{\beta,\textbf{k}}(T)(x)$ converge a.u. (b.a.u.) to $\widehat{x}\in L_0$ if and only if the averages $M_n^{\gamma\beta}(T)(x)$ converge a.u. (respectively b.a.u.) to $d\widehat{x}$.
\end{proposition}
\begin{proof}
Assume that the averages $M_n^{\beta,\textbf{k}}(T)(x)$ converge a.u. (b.a.u.) to $\widehat{x}$. Assume, without loss of generality, that $k_0=0$, and let $c(n)=\operatorname{card}(\{k_j:j\in\mathbb{N}_0\}\cap\{0,...,n\})$. Observe that $c(n)=\max\{j:k_j\leq n\}+1$, $n\in\mathbb{N}_0$. It follows that, since
\[
\begin{split}
M_n^{\gamma\beta}(T)(x)&=\frac1n\sum_{j=0}^{n-1}\gamma_j\beta_jT^j(x)=\frac1n\sum_{j=0}^{c(n-1)-1}\beta_{k_j}T^{k_j}(x)\\
&=\frac{c(n-1)}n\frac1{c(n-1)}\sum_{j=0}^{c(n-1)-1}\beta_{k_j}T^{k_j}(x)=\frac{c(n-1)}{n} M_{c(n-1)}^{\beta,\textbf{k}}(T)(x),
\end{split}
\]
$\frac{c(n-1)}{n}\to d$, and the subsequence $\{M_{c(n-1)}^{\beta,\textbf{k}}(T)(x)\}_{n=1}^{\infty}$ converges a.u. (b.a.u.) to $\widehat{x}$, Lemma 3.2 implies that the sequence $\{M_n^{\gamma\beta}(T)(x)\}_{n=1}^{\infty}$ converges a.u. (b.a.u.) to $d\widehat{x}$.

Similarly,
\[
\begin{split}
M_n^{\beta,\textbf{k}}(T)(x)&=\frac{1}{n}\sum_{j=0}^{n-1}\beta_{k_j}T^{k_j}(x)=\frac{1}{n}\sum_{j=0}^{k_{n-1}}c_j\beta_jT^j(x)\\
&=\frac{k_{n-1}+1}{n}\frac1{k_{n-1}+1}\sum_{j=0}^{k_{n-1}}c_j\beta_jT^j(x)=\frac{k_{n-1}+1}n M_{k_{n-1}+1}^{\gamma\beta}(T)(x),
\end{split}
\]
implies the converse.
\end{proof}

\medskip

\begin{theorem}\label{t31} Let $\textbf{k}=\{k_j\}_{j=0}^{\infty}$ have density $1$, $\beta=\{\beta_j\}_{j=0}^\infty\in\ell_\infty(\mathbb{C})$, and $2\leq p<\infty$ ($1\leq p<\infty$). If $T\in DS^+$ and $\beta$ is a.u.- (respectively, b.a.u.-) good for $T$ on $L_p(\mathcal{M},\tau)$, then $\textbf{k}$ is a.u.- (respectively, b.a.u.-) good for $T$ and $\beta$ on $L_p(\mathcal{M},\tau)$. Moreover, for every $x\in L_p$, the sequence $\{M_n^{\beta,\textbf{k}}(T)(x)\}_{n=1}^{\infty}$ converges a.u. (respectively, b.a.u.) to some $\widehat{x}\in L_p$.
\end{theorem}

\begin{proof}
We will prove the a.u. case, as the b.a.u. case will follow the same line of proof with the appropriate modifications. As such, assume $2\leq p<\infty$.


Assume $x\in L_1\cap\mathcal{M}$. Then, since $\beta$ is a.u.-good for $T$ on $L_p(\mathcal{M},\tau)$ and since $L_1\cap\mathcal{M}\subset L_p$, we see that the sequence $\{M_n^\beta(T)(x)\}_{n=1}^{\infty}$, hence the subsequence $\{M_{k_n}^{\beta}(T)(x)\}_{n=1}^{\infty}$, converges a.u.

Let $A_n(T)(x)=\frac{1}{k_n}\sum_{j=0}^{n-1}\beta_{k_j}T^{k_j}(x),$ so that $M_n^{\beta,\textbf{k}}(T)(x)=\frac{k_n}{n}A_n(T)(x)$. Since $T$ is a contraction on $\mathcal{M}$, we see that $\|T^j(x)\|_\infty\leq\|x\|_\infty$ for every $j\geq0$. Let $\|\beta\|_{\ell_\infty}=\sup_{j\geq0}|\beta_j|$. Then, for each $n\geq1$,
\begin{align*}
\left\|M_{k_n}^{\beta}(T)(x)-A_n(T)(x)\right\|_\infty
&=\left\|\frac{1}{k_n}\sum_{j=0,\; j\not\in\textbf k}^{k_n-1}\beta_jT^j(x)\right\|_\infty \\
&\leq\frac{1}{k_n}\sum_{j=0,\; j\not\in\textbf k}^{k_n-1}|\beta_j|\|T^j(x)\|_\infty \\
&\leq\frac{1}{k_n}\sum_{j=0,\; j\not\in\textbf k}^{k_n-1}\|\beta\|_{\ell_\infty}\|x\|_\infty
= \frac{k_n-n}{k_n}\|\beta\|_{\ell_\infty}\|x\|_\infty.
\end{align*}
Let $\epsilon>0$. Since $\frac{k_n-n}{k_n}\to0$ as $n\to\infty$ (see Remark \ref{r31}), there exists $N\in\mathbb N$ such that $n\geq N$ implies 
\[
\|M_{k_n}^{\beta}(T)(x)-A_n(T)(x)\|_\infty<\epsilon.
\] 
Thus, it follows by Lemma 3.1 that the sequence $\{A_n(T)(x)\}_{n=1}^{\infty}$ converges a.u. since $\{M_{k_n}^{\beta}(T)(x)\}_{n=1}^{\infty}$ does.

Since $L_0$ is complete with respect to a.u. (and b.a.u.) convergence (see Theorem 2.3 and Remark 2.4 of \cite{cls}), there exists $\widehat{x}\in L_0$ such that $A_n(T)(x)\to\widehat x$ a.u. Therefore, since $\frac{k_n}{n}\to1$ (Remark \ref{r31}) and $M_n^{\beta,\textbf{k}}(T)(x)=\frac{k_n}{n}A_n(T)(x)$, we conclude, by Lemma \ref{l32}, that $\{M_n^{\beta,\textbf{k}}(T)(x)\}_{n=1}^{\infty}$ converges a.u. for every $x\in L_1\cap\mathcal{M}$.

As $\{\frac{k_n}{n}\}_{n=1}^{\infty}$ and $\beta$ are a bounded sequences, $\{M_n^{\beta,\textbf{k}}(T)\}_{n=1}^{\infty}$ is u.e.m. at zero on $(L_p,\|\cdot\|_p)$ by Proposition \ref{p31}. This in turn implies,  by Proposition \ref{p22}, that the set
\[
C=\{x\in L_p(\mathcal{M},\tau):\{M_n^{\beta,\textbf{k}}(T)(x)\}_{n=1}^{\infty}\ \text{ converges a.u.}\}
\]
is closed in $L_p$. Since $L_1\cap \mathcal{M}$ is dense in $L_p$ it follows that $L_p=C$.

Therefore, given $x\in L_p$, there exists $\widehat{x}\in L_0$ such that $M_n^{\beta,\textbf{k}}(T)(x)\to\widehat{x}$ a.u., hence in measure. Since the unit ball in $L_p$ is closed with respect to the measure topology (see, for example, Theorem 1.2 in \cite{cls}) and $\|M_n^{\beta,\textbf{k}}(T)(x)\|_p\leq\|\beta\|_{\ell_\infty}\|x\|_p$ for all $n$, we conclude that $\widehat{x}\in L_p$.
\end{proof}

If $\beta=\textbf{1}$, this establishes the following.

\begin{corollary}\label{c31}
If $\textbf{k}=\{k_j\}_{j=0}^{\infty}$ has density $1$ and $x\in L_p$, then the subsequential averages $M_n^{\textbf{k}}(T)(x)$ converge b.a.u. (a.u.) to some $\widehat{x}\in L_p$ if $1\leq p<\infty$ (respectively, if $2\leq p<\infty$).
\end{corollary}

Combining Theorem \ref{t31} with Theorem 3.1 of \cite{cl1}, we arrive at the following.

\begin{corollary}\label{c32}
Assume $\mathcal{M}$ has a separable predual. If $\textbf{k}$ has density $1$ and $\beta$ is a bounded Besicovich sequence, then, given $x\in L_p$, the subsequential weighted averages $M_n^{\beta,\textbf{k}}(T)(x)$ converge b.a.u. (a.u.) to some $\widehat{x}\in L_p$ if $1\leq p<\infty$ (respectively, if $2\leq p<\infty$.
\end{corollary}

Suppose $\alpha,\beta\in\ell_\infty(\mathbb{C})$ are such that there exists a set $E\subset\mathbb{N}_0$ with density $1$ such that $\alpha_j\neq\beta_j$ implies $j\in E^c$. Then, similar to the proof of Theorem 3.1, it follows that for every $x\in L_1\cap\mathcal{M}$ and $n\in\mathbb N$,
\begin{align*}
\|M_n^\alpha(T)(x)&-M_n^\beta(T)(x)\|_\infty
\leq\frac{1}{n}\sum_{j=0,\; j\not\in E}^{n-1}|\alpha_j-\beta_j|\|T^j(x)\|_\infty \\
&\leq\frac{n-\operatorname{card}(\{0,...,n-1\}\cap E)}{n}(\|\alpha\|_{\ell_\infty}+\|\beta\|_{\ell_\infty})\|x\|_\infty,
\end{align*}
which tends to $0$ as $n\to\infty$. Since both weighted averages are b.u.e.m. at zero on $L_p$ for $1\leq p<\infty$ and u.e.m. at zero on $L_p$ for $2\leq p<\infty$, an application of Lemma \ref{l31} and Proposition \ref{p22} then implies the following.

\medskip

\begin{theorem}\label{t32}
Let $\alpha,\beta\in\ell_\infty(\mathbb C)$ be such that there exists a set $E\subseteq\mathbb{N}_0$ with density $1$ such that $\alpha_j\neq\beta_j$ implies $j\in E^c$. Then for $2\leq p<\infty$ ($1\leq p<\infty$) and $T\in DS^+$, $\alpha$ is a.u.- (respectively b.a.u.-) good for $T$ on $L_p(\mathcal{M},\tau)$ if and only if $\beta$ is a.u.- (respectively b.a.u.-) good for $T$ on $L_p(\mathcal{M},\tau)$.
\end{theorem}

\begin{remark}\label{r32}
\textbf{1.} Theorem \ref{t32} together with Proposition \ref{p32} shows that the conclusion of Theorem 3.1 is actually an equivalence.

\textbf{2.} The proof of Theorem \ref{t31} also shows that the a.u. limit of $\{M_n^{\beta,\textbf{k}}(T)(x)\}_{n=1}^{\infty}$ is the same as that of $\{M_n^\beta(T)(x)\}_{n=1}^{\infty}$ whenever $x\in L_1\cap\mathcal{M}$. To see this, let $x\in L_1\cap\mathcal{M}$ and let $\widehat{x}\in L_0$ be the a.u. limit of $\{M_n^\beta(T)(x)\}_{n=1}^{\infty}$. Then, for any $\epsilon>0$ with corresponding projection $e$ from the definition of the a.u. convergence of $\{M_{k_n}^\beta(T)(x)\}_{n=1}^{\infty}$ to $\widehat{x}$, we find
$$\|(A_n(x)-\widehat{x})e\|_\infty\leq\|A_n^\beta(x)-M_{k_n}^\beta(T)(x)\|_\infty+\|(M_{k_n}^\beta(T)(x)-\widehat{x})e\|_\infty\xrightarrow{n\to\infty}0,$$ which shows that $\{A_n(x)\}_{n=1}^{\infty}$ shares its a.u. limit with $\{M_{k_n}^{\beta}(T)(x)\}_{n=1}^{\infty}$. The rest of the proof is similar to that of Theorem \ref{t31}. The same can be said for b.a.u. convergence, as well as the convergences in Theorem \ref{t32} for $x\in L_1\cap\mathcal{M}$.
\end{remark}

\section{Convergence along Uniform Sequences}

We will now extend the main result of \cite{lm} regarding the convergence of subsequential averages along uniform sequences to the general non-commutative $L_p$ setting for every $1\leq p<\infty$. Furthermore, we also show that we can replace the b.a.u. convergence as stated there with a.u. convergence for $2\leq p<\infty$, as well give a condition that allows one to add weights to these averages.

We will first give a very brief review of the terminology used in that paper, namely that relating to uniform sequences. See \cite{bk} and \cite{sat} for more details and properties of uniform sequences.

Let $X$ be a compact Hausdorff space, and let $\phi:X\to X$ be a continuous map such that the family $\{\phi^n\}_{n=0}^{\infty}$ is equicontinuous.

The pair $(X,\phi)$ is called \textit{strictly ergodic} if there exists a unique $\phi$-invariant measure $\mu$ on $(X,\mathcal{B}_X)$ (so that $\mu\circ\phi=\mu$), where $\mathcal{B}_X$ denotes the $\sigma$-algebra of all Borel subsets of $X$, such that $\operatorname{supp}(\mu)=X$ and such that for any $z\in X$ and $f\in C(X)$ we have
\[
\int_X f(t)d\mu(t)=\lim_{n\to\infty}\frac{1}{n}\sum_{k=0}^{n-1}f(\phi^k(z)),
\]
where $C(X)$ is the Banach space of continuous functions $f:X\to\mathbb{C}$ with the supremum norm, and the convergence above is with respect to that norm.

\smallskip

Fix a strictly ergodic system $(X,\mathcal{B}_X,\mu,\phi)$ and a set $Y\in \mathcal{B}_X$ such that $\mu(Y)>0=\mu(\partial Y)$, where $\partial Y$ denotes the boundary of $Y$. A sequence $\textbf{k}=\{k_j\}_{j=0}^{\infty}$ of non-negative integers is said to be \textit{uniform} if there exists a point $\omega_0\in X$ such that

$k_0=\min\{k\geq0:\phi^k(\omega_0)\in Y\},$

$k_n=\min\{k>k_{n-1}:\phi^k(\omega_0)\in Y\}, n\geq1.$

\noindent The triple $(X,\mathcal{B}_X,\mu,\phi)$, $Y$, and $\omega_0$ will be called the \textit{apparatus} for $\textbf{k}$ [BK].

\smallskip

For the rest of this section $\textbf{k}=\{k_j\}_{j=0}^{\infty}$ will be a fixed uniform sequence with the apparatus $(X,\mathcal{B}_X,\mu,\phi)$, $Y$, and $\omega_0$. Details of the next two lemmas can be found in \cite{sat}.

\begin{lemma}\label{l41}
If $\{k_j\}_{j=0}^{\infty}$ is a uniform sequence with the apparatus above, then
$$\lim_{n\to\infty}\frac{n}{k_n}=\mu(Y).$$
\end{lemma}

\begin{lemma}\label{l42}
If $\{k_j\}_{j=0}^{\infty}$ is a uniform sequence with the apparatus as given above, then for any $\epsilon>0$ there exist open subsets $Y_1$, $Y_2$, and $W$ of $X$ such that

(i) $Y_1\subseteq Y\subseteq Y_2$, $\mu(Y_2\setminus Y_1)<\epsilon$, and $\mu(\partial Y_1)=\mu(\partial Y_2)=0$.

(ii) $\omega_0\in W$ and, for every $\omega\in W$ and all $k\geq0$,
\[
\chi_{Y_1}(\phi^k(\omega))\leq\chi_{Y}(\phi^k(\omega_0)\leq\chi_{Y_2}(\phi^k(\omega)).
\]
\end{lemma}

The following is the main result of \cite{lm}.

\begin{theorem}\label{t41}\cite{lm} Let a semifinite von Neumann algebra $(\mathcal{M},\tau)$ have a separable predual, and let $T\in DS^+(\mathcal{M},\tau)$. If $\textbf{k}$ is a uniform sequence and $x\in L_1(\mathcal{M},\tau)$, then the subsequential averages $M_n^{\textbf{k}}(T)(x)$ converge b.a.u. to some $\widehat{x}\in L_1(\mathcal{M},\tau)$.
\end{theorem}


As they will be necessary for the statement and proof of the next theorem, we will let $\widetilde{L}_p$ and $\widetilde{\mathcal{M}}$ denote the spaces of Bochner $p$-integrable and bounded weakly measurable functions respectively (where $1\leq p<\infty$) on the measure space $(X,\mathcal{B}_X,\mu)$. If $\mathcal{M}$ has n.s.f. trace $\tau$, then $\widetilde{\mathcal{M}}$ is a semifinite von Neumann algebra with n.s.f. trace $\widetilde{\tau}$ given by
$$
\widetilde{\tau}(f)=\int_{X}\tau(f(z))d\mu(z)
$$
whenever $f\in\widetilde{\mathcal{M}}^+$. If $\mathcal{M}$ has a separable predual then, in view of Proposition 1.22.12 and Theorem 1.22.13 of \cite{sak}, we see that predual $L_1(\widetilde{\mathcal{M}},\widetilde{\tau})$ of $\widetilde{M}$ is isomorphic to $\widetilde{L_1}$, and $L_p(\widetilde{M},\widetilde{\tau})$ is isomorphic to $\widetilde{L}_p$ for $1< p<\infty$.

\medskip

\begin{theorem}\label{t42} Assume $\mathcal{M}$ has a separable predual. Let $\textbf{k}=\{k_j\}_{j=0}^{\infty}$ be a uniform sequence and $\beta\in\ell_\infty(\mathbb C)$. If $\beta$ is a.u.- (b.a.u.-) good in $L_p(\widetilde{\mathcal{M}},\widetilde{\tau})$ for $2\leq p<\infty$ (respectively, $1\leq p<\infty$), then $\textbf{k}$ is a.u.- (respectively, b.a.u.-) good for $\beta$ in $L_p(\mathcal{M},\tau)$, with the limit $\widehat{x}\in L_p(\mathcal{M},\tau)$.
\end{theorem}
\begin{proof}
Our proof follows closely the proof of Theorem 4 of \cite{lm}. We will only prove the a.u. case. The b.a.u. case will follow similarly. Accordingly, assume that $2\leq p<\infty$ and $T\in DS^+(\mathcal{M},\tau)$.

Define $\widetilde{T}:\widetilde{L}_1+\widetilde{\mathcal{M}}\to\widetilde{L}_1+\widetilde{\mathcal{M}}$ by, for $z\in X$,
$$(\widetilde{T}f)(z)=T(f(\phi(z))).$$
Then $\widetilde{T}\in DS^+(\widetilde{M},\widetilde{\tau})$, and so we may consider the corresponding weighted averages
$$
\frac{1}{n}\sum_{j=0}^{n-1}\beta_j\widetilde{T}^j(f)= \frac{1}{n}\sum_{j=0}^{n-1}\beta_jT^j(f\circ\phi^j)
$$
for $f\in\widetilde{L}_1+\widetilde{\mathcal{M}}$. Since $\beta$ is a.u.-good in $L_p(\widetilde{M},\widetilde{\tau})$, it follows that the averages $
\{\frac{1}{n}\sum_{j=0}^{n-1}\beta_j\widetilde{T}^j(f)\}_{n=1}^{\infty}$ converge a.u. for every $f\in\widetilde{L}_1\cap\widetilde{M}$. By Lemma 4.1 in \cite{cls}, it follows that
$$ \frac{1}{n}\sum_{j=0}^{n-1}\beta_jT^j(f(\phi^j(z)))
$$
converge a.u. for $\mu$-a.e. $z\in X$.

Assume $x\in L_1\cap\mathcal{M}$ and fix $\epsilon>0$. Let $W,Y_1,Y_2\subset X$ satisfy the conclusion of Lemma \ref{l42} with this $\epsilon$. Define $\widetilde{x}:X\to\widetilde{L}_1$ by 
\[
\widetilde{x}(z)=\chi_{Y_1}(z)x
\]
for $z\in X$. Then $\widetilde{x}\in\widetilde{L}_1\cap\widetilde{\mathcal{M}}$, so that the averages
$$\frac{1}{n}\sum_{j=0}^{n-1}\beta_jT^j(x)\chi_{Y_1}(\phi^j(z))
$$
converge a.u. for $\mu$-a.e. $z\in X$. Since $\mu(W)>0$, we can find $\omega\in W$ such that
\[
\frac{1}{n}\sum_{j=0}^{n-1}\beta_jT^j(x)\chi_{Y_1}(\phi^j(\omega))
\]
converges a.u.

Lemma 1 of \cite{sat} tells us that, since $\mu(\partial Y_j)=0$, we have
\[
\lim_{n\to\infty}\frac{1}{n}\sum_{j=0}^{n-1}\chi_{Y_k}(\phi^j(\omega))=\mu(Y_k)\ \text{ for }k\in\{1,2\}.
\]
The properties of $Y_1$, $Y_2$, and $W$ from Lemma \ref{l42} then entail that there exists $N$ such that $n\geq N$ implies
$$0\leq\frac{1}{n}\sum_{j=0}^{n-1}[\chi_Y(\phi^j(\omega_0))-\chi_{Y_1}(\phi^j(\omega))]<\epsilon.$$
With this, letting $\|\beta\|_{\ell_\infty}=\sup_{j\geq0}|\beta_j|$, we find that
\begin{align*}
\left\|\frac{1}{n}\sum_{j=0}^{n-1}\beta_jT^j(x)\chi_Y(\phi^j(\omega_0))\right.&\left.-\frac{1}{n}\sum_{j=0}^{n-1}\beta_jT^j(x)\chi_{Y_1}(\phi^j(\omega))\right\|_\infty \\
&\leq\frac{1}{n}\sum_{j=0}^{n-1}\|\beta_jT^j(x)(\chi_Y(\phi^j(\omega_0))-\chi_{Y_1}(\phi^j(\omega))\|_\infty \\
&\leq \|\beta\|_{\ell_\infty}\|x\|_\infty\frac{1}{n}\sum_{j=0}^{n-1}(\chi_{Y}(\phi^j(\omega_0))-\chi_{Y_1}(\phi^j(\omega))) \\
&<\|\beta\|_{\ell_\infty}\|x\|_\infty\epsilon,
\end{align*}
which, by Lemma \ref{l31}, implies that the averages
\[
\frac{1}{n}\sum_{j=0}^{n-1}\beta_jT^j(x)\chi_Y(\phi^j(\omega_0))\]
converge a.u. as $n\to\infty$.

Next, we have
\[
M_n^{\beta,\textbf{k}}(T)(x)=\frac{1}{n}\sum_{j=0}^{n-1}\beta_{k_j}T^{k_j}(x)=\frac{k_{n-1}+1}{n}\frac{1}{k_{n-1}+1}\sum_{j=0}^{k_{n-1}}\beta_jT^j(x)\chi_{Y}(\phi^j(\omega_0)).
\]
Since $\frac{k_{n-1}+1}{n}\to\frac{1}{\mu(Y)}$ as $n\to\infty$ by Lemma \ref{l41}, Lemma \ref{l32} implies that the averages $\{M_n^{\beta,\textbf{k}}(T)(x)\}_{n=1}^{\infty}$ converge a.u. for every $x\in L_1\cap\mathcal{M}$.

As the sequences $\{\frac{k_n}{n}\}_{n=1}^{\infty}$ and $\beta$ are bounded, Proposition \ref{p31} implies that $\{M_n^{\beta,\textbf{k}}(T)\}_{n=1}^{\infty}$ is u.e.m. at zero on $(L_p,\|\cdot\|_p)$. By Proposition \ref{p22}, this fact implies that the set
$$C=\{x\in L_p:\{M_n^{\beta,\textbf{k}}(T)(x)\}_{n=1}^{\infty}\text{ converges a.u.}\}$$ is closed in $L_p$. Since $L_1\cap\mathcal{M}$ is a subset of $C$ which is dense in $L_p$, it follows that $L_p=C$. Therefore, since $L_0$ is complete with respect to a.u. convergence, given $x\in L_p$, there is $\widehat{x}\in L_0$ such that $M_n^{\beta,\textbf{k}}(T)(x)\to\widehat{x}$ a.u. Finally, since $\|M_n^{\beta,\textbf{k}}(T)(x)\|_p\leq\|\beta\|_{\ell_\infty}\|x\|_\infty$, we conclude as in Theorem \ref{t31} that $\widehat{x}\in L_p$.
\end{proof}

Similar to what we did in the density $1$ case, by setting $\beta=\textbf{1}$, this yields a  generalization of Theorem \ref{t41}.

\begin{corollary}\label{c41} Assume $\mathcal{M}$ has a separable predual and $T\in DS^+$. If $\textbf{k}$ is a uniform sequence, then, given $1\leq p<\infty$, the subsequential averages $M_n^{\textbf{k}}(T)(x)$ converge b.a.u. to some $\widehat{x}\in L_p(\mathcal{M},\tau)$ for every $x\in L_p(\mathcal{M},\tau)$. Furthermore, this convergence occurs a.u. when $2\leq p<\infty$.
\end{corollary}

\begin{remark}\label{r41} \textbf{1.}  It is known that $\mathcal{M}$ having a separable predual is equivalent to being able to assume that the Hilbert space $\mathcal{H}$ it acts on which is separable (see \cite{su}). In this case it is also known that $\widetilde{M}$ is $*$-isomorphic to $L_\infty(X,\mathcal{B}_X,\mu)\overline{\otimes}\mathcal{M}$, and that this acts on $L_2(X,\mathcal{B}_X,\mu)\otimes\mathcal{H}$. If $L_1(X,\mathcal{B}_{X},\mu)$ is separable, then $L_2(X,\mathcal{B}_{X},\mu)$ is too, so that the tensor product Hilbert space $L_2(X,\mathcal{B}_{X},\mu)\otimes\mathcal{H}$ is separable. Therefore $\widetilde{M}$ acts on a separable Hilbert space, implying it has a separable predual. So if $(X,\mathcal{B}_X,\mu)$ is such that $L_1(X,\mathcal{B}_X,\mu)$ is separable, then Theorem 4.2 will hold for every bounded Besicovich sequence.

\textbf{2.} Lemma 4.1 can be shown to be equivalent to saying that a uniform sequence $\textbf{k}$ with apparatus $(X,\mathcal{B}_X,\mu)$, $Y$, and $\omega_0$ has  density $\mu(Y)$.
\end{remark}

\section{Convergence along Block Sequences of Positive Lower Density}

In this section, we will study the a.u. and b.a.u. convergence of the subsequential averages along certain types of block sequences with positive lower density. After this, we show that this can be extended by allowing bounded Besicovich sequences as weights when $\mathcal{M}$ has a separable predual. For more information about block sequences, see \cite{bl} or \cite{ps}. The norm convergence of the subsequential ergodic averages of any $T\in DS^+$ along block sequences was considered in \cite{comli}.

Let $\mathcal{I}=\{I_n=[a_n,b_n]\}_{n=0}^{\infty}$ be a sequence of intervals in $[0,\infty)$ whose endpoints are in $\mathbb{N}_0$ such that $b_n<a_{n+1}$ for every $n\geq0$.

If $\textbf{k}=\{k_n\}_{n=0}^{\infty}$ is the sequence determined by the strictly increasing enumeration of $\bigcup_{n}I_n\cap\mathbb{N}_0$, then $\textbf{k}$ is called a \textit{block sequence}. Define a function $N_{\mathcal{I}}:\mathbb{N}_0\to\mathbb{N}_0$ to be such that $$k_n\in[a_{N_{\mathcal{I}}(n)},b_{N_{\mathcal{I}}(n)}]$$ (which exists and is well-defined by the restrictions on the intervals). Observe that $N_{\mathcal{I}}(n)$ counts the number of intervals of $\mathcal{I}$ one has to skip before finding the interval that $k_n$ lives in.

\begin{example}\label{e51}
For each $n\in\mathbb{N}_0$, let $I_n=[n^2,n^2+n]$. Let $\mathcal{I}=\{I_n\}_{n=0}^{\infty}$, and let $\textbf{k}$ be the strictly increasing enumeration of $\bigcup_{n=0}^{\infty}I_n\cap\mathbb{N}_0$, so that $\textbf{k}$ is a block sequence. One can show that $\textbf{k}$ has lower density $\frac{1}{2}$ and that $N_{\mathcal{I}}(n)=[\sqrt{n}]$, where $[x]$ denotes the greatest integer less than or equal to $x$ for each $x\in\mathbb{R}$, which implies that $0\leq\frac{N_{\mathcal{I}}(n)}{n}\leq\frac{1}{\sqrt{n}}\to0$ as $n\to\infty$.
\end{example}

\begin{theorem}\label{t51}
Let $\textbf{k}=\{k_j\}_{j=0}^{\infty}$ be a block sequence with lower density $d>0$ such that $\frac{N_{\mathcal{I}}(n)}{n}\to0$ as $n\to\infty$ and $T\in DS^+$. Then, given $x\in L_p(\mathcal{M},\tau)$, the subsequential averages
$M_n^{\textbf{k}}(T)(x)$ converge a.u. (b.a.u.) to some $\widehat{x}\in L_p(\mathcal{M},\tau)$ if $2\leq p<\infty$ (respectively, $1\leq p<\infty$).
\end{theorem}
\begin{proof}
First we will prove the claim for $x\in L_2(\mathcal{M},\tau)$. Since $T$ is a contraction on all of $L_2$, and not just its self-adjoint part, the set
$$\{x\in L_2:T(x)=x\}+\{x=y-T(y):y\in L_2\cap\mathcal{M}\}$$ is dense in $L_2$ (see Theorem 4.1 in \cite{li}). If $x\in L_2$ is such that $T(x)=x$, then $M_n^{\textbf{k}}(T)(x)=x$ for every $n$. Since for every $\epsilon>0$ there exists $e\in\mathcal{P}(\mathcal{M})$ such that $(M_n^{\textbf{k}}(T)(x)-x)e\in\mathcal{M}$ for every $n$ and $\tau(e^\perp)\leq\epsilon$, we conclude that $\{M_n^{\textbf{k}}(T)(x)\}_{n=1}^{\infty}$ converges a.u. to $x$.
	
Now assume that $x=y-T(y)$ for some $y\in L_2\cap\mathcal{M}$. Fix $n\geq1$. Then
\begin{align*}
M_n^{\textbf{k}}&(T)(x)
=\frac{1}{n}\sum_{j=0}^{n-1}T^{k_j}(y-T(y))=\frac{1}{n}\sum_{j=0}^{n-1}\left(T^{k_j}(y)-T^{k_j+1}(y)\right) \\
&=\frac{1}{n}\left((T^{a_{N_{\mathcal{I}}(n-1)}}(y)-T^{k_{n-1}+1}(y))+\sum_{i=0}^{N_{\mathcal{I}}(n-1)-1}\sum_{j=a_i}^{b_i}\left(T^{j}(y)-T^{j+1}(y)\right)\right) \\
&=\frac{1}{n}\left((T^{a_{N_{\mathcal{I}}(n-1)}}(y)-T^{k_{n-1}+1}(y))+\sum_{i=0}^{N_{\mathcal{I}}(n-1)-1}\left(T^{a_i}(y)-T^{b_i+1}(y)\right)\right) \\
\end{align*}
and since $T$ is a contraction on $\mathcal{M}$, we obtain
$$
\left\|M_n^{\textbf{k}}(T)(x)\right\|_\infty
\leq2\|y\|_\infty\frac{N_{\mathcal{I}}(n-1)+1}{n}
.$$
By assumption we know that $\lim_{n\to\infty}\frac{N_{\textbf{k}}(n)}{n}=0$, implying $\|M_n^{\textbf{k}}(T)(x)\|_\infty\to0$ as $n\to\infty$. Therefore $\{M_n^{\textbf{k}}(T)(x)\}_{n=1}^{\infty}$ converges uniformly, and so a.u., as $n\to\infty$.
	
Since $\textbf{k}$ has lower density $d>0$, $\{M_n^{\textbf{k}}(T)\}_{n=1}^{\infty}$ is u.e.m. at zero on $(L_2,\|\cdot\|_2)$ by Proposition \ref{p31}. Therefore, since $\{M_n^{\textbf{k}}(T)(x)\}_{n=1}^{\infty}$ converges a.u. on dense subset of $L_2$, it follows that $\{M_n^{\textbf{k}}(T)(x)\}_{n=1}^{\infty}$ converges a.u. for every $x\in L_2$ by Proposition \ref{p22}. In particular, this convergence holds for every $x\in L_1\cap\mathcal{M}$.
	
Since the averages $\{M_n^{\textbf{k}}(T)\}_{n=1}^{\infty}$ are also u.e.m. (b.u.e.m.) at zero on $(L_p,\|\cdot\|_p)$ for $2\leq p<\infty$ (respectively, $1\leq p<\infty$) by Proposition \ref{p31}, since a.u. convergence implies b.a.u. convergence, and since $L_1\cap\mathcal{M}$ is dense in $L_p$, it follows by Proposition \ref{p22} that $\{M_n^{\textbf{k}}(T)(x)\}_{n=1}^{\infty}$ converges a.u. (respectively, b.a.u.) for every $x\in L_p$ when $2\leq p<\infty$ (respectively, $1\leq p<\infty$). Finally, we observe as in the proof of Theorem \ref{t31}, that the limit $\widehat{x}\in L_p$.
\end{proof}

As mentioned in the proof, the proof of Theorem \ref{t42} above very closely follows Theorem 4 of \cite{lm}. One also sees that the proof of Theorem 4 of \cite{lm} is similar to the proof of Lemma 4.2 in \cite{cls}, which itself implies Theorem 4.6 of \cite{cls} and Theorem 3.1 in \cite{cl1}. Following the commonalities in the core parts of the proofs of these statements, we omit the proof of the following extension of Theorem \ref{t51} since it follows the same lines.

\begin{theorem}\label{t52} Let $\mathcal{M}$ have a separable predual and $2\leq p<\infty$ ($1\leq p<\infty$). If $\beta$ is a bounded Besicovich sequence and $\textbf{k}$ is a block sequence with lower density $d>0$ such that $\frac{N_{\mathcal{I}}(n)}{n}\to0$, then $\textbf{k}$ is a.u.- (respectively, b.a.u.-) good for $\beta$ in $L_p(\mathcal{M},\tau)$, with the limits  in $L_p(\mathcal{M},\tau)$.
\end{theorem}

\noindent
{\bf Acknowledgements:} The author is grateful to Dr. Semyon Litvinov for his help in simplifying and improving numerous aspects of this paper. The author is also thankful to Dr. Do\u{g}an \c{C}\"{o}mez for his guidance and support with writing this paper.

\end{document}